\documentclass[12pt]{article}

\usepackage[a4paper,margin=1in]{geometry}
\usepackage{amsmath,amssymb,amsthm,amsfonts,mathtools,mathrsfs,bm}
\usepackage{booktabs,array,multirow}
\usepackage{enumitem}
\usepackage{hyperref}
\usepackage{tikz}
\usepackage{algorithm}
\usepackage{algpseudocode}
\usepackage{caption}
\usepackage{float}
\usepackage{graphicx}
\renewenvironment{proof}{{\bf \noindent Proof.}}{\qed}
\everymath{\displaystyle}

\usetikzlibrary{arrows.meta,positioning,calc,shapes.geometric,fit,backgrounds}

\hypersetup{
	colorlinks=true,
	linkcolor=blue,
	citecolor=blue,
	urlcolor=blue
}

\newtheorem{theorem}{Theorem}[section]
\newtheorem{lemma}[theorem]{Lemma}
\newtheorem{proposition}[theorem]{Proposition}
\newtheorem{corollary}[theorem]{Corollary}

\newtheorem{openproblem}[theorem]{Open Problem}
\theoremstyle{definition}
\newtheorem{definition}[theorem]{Definition}
\newtheorem{example}[theorem]{Example}
\theoremstyle{remark}

\DeclareMathOperator{\reg}{reg}
\DeclareMathOperator{\pd}{pd}
\DeclareMathOperator{\comp}{comp}
\DeclareMathOperator{\Tor}{Tor}
\DeclareMathOperator{\Hilb}{Hilb}

\newcommand{\kk}{\mathbb{K}}
\newcommand{\ZZ}{\mathbb{Z}}
\newcommand{\EE}{\mathcal{E}}
\newcommand{\MM}{\mathcal{M}}
\newcommand{\CC}{\mathcal{C}}

\newcommand{\ind}{\mathrm{Ind}}
\newcommand{\bk}[1]{\mathbf{#1}}

\newcommand{\wt}{\widetilde}

\title{Graded Betti numbers of generalized split--join graphs and applications}

\author{Bilal Ahmad Rather\\[2mm]
	\small Department of Mathematics, Samarkand International University of Technology,\\
	\small Samarkand 140100, Uzbekistan\\
	\texttt{\href{mailto:bilalahmadrr@gmail.com}{bilalahmadrr@gmail.com}}
}

\date{}

\begin{document}
	
	\maketitle
	
	\begin{abstract}
		We determine the full graded Betti tables of graph families that subsume several classes studied recently in the literature, namely the generalized multiple complete split-like graphs and the generalized clique-star graphs with arbitrary clique block sizes. The method combines Hochster's formula with a precise decomposition of the associated independence complexes into disjoint unions of simplices and iterated joins of discrete complexes. This reduces every graded Betti number to an explicit coefficient extraction formula and yields closed expressions for the linear strand, higher strands, Hilbert series, regularity, projective dimension, and extremal Betti numbers. In particular, we prove a sharp criterion for $2$-linear resolution and identify the regularity corner in terms of the number of nontrivial clique blocks. As applications, we recover and extend earlier results on equal-block split-like graphs, obtain complete formulas for pineapple graphs, and derive consequences for power graphs of cyclic groups, elementary abelian groups, and prime-power dihedral groups. 
	\end{abstract}
	
	
	
	\section{Introduction}\label{sec:intro}
	
	The interaction between combinatorics and commutative algebra through edge ideals has become one of the most active themes in modern algebraic combinatorics. To a finite simple graph $G$ with vertex set $V(G)=\{x_1,\dots,x_N\}$, we associates the squarefree quadratic monomial ideal
	$$
	I(G)=(x_ix_j:{x_i,x_j}\in E(G))\subseteq R=\kk[x_1,\dots,x_N],
	$$
	and the quotient ring $R/I(G)$ records in algebraic form many structural features of the graph. The graded Betti numbers
	$$
	\beta_{i,j}(G):=\dim_{\kk}\Tor_i^R(\kk,R/I(G))_j
	$$
	describe the minimal graded free resolution of $R/I(G)$, and hence control several fundamental invariants such as the Castelnuovo--Mumford regularity and the projective dimension. Classical references for this viewpoint include Villarreal, Jacques, Herzog and Hibi, and the works surrounding Hochster's formula and Fröberg's theorem \cite{Villarreal1995,Hochster1977,Jacques2004,HerzogHibi2010,Froberg1990}. More recent developments emphasize exact Betti tables for structured graph classes, regularity bounds, extremal Betti numbers, and computational consequences \cite{BayerCharalambousPopescu1999,CorsoNagel2009,HaVanTuyl2008,FernandezRamosGimenez2014,Woodroofe2014,Froberg2022,Froberg2023}.
	
	From the graph-theoretic side, split graphs and graphs assembled by graph joins are especially attractive, since they sit between rigid and flexible behavior. They are simple enough to admit explicit combinatorial descriptions, yet rich enough to model algebraic graphs arising from groups and rings. In the recent paper \cite{Rather2024}, initial Betti numbers of multiple complete split-like graphs, clique-star graphs, and some equal-block generalizations were determined, together with last-column extremal Betti numbers and projective dimensions. The same paper showed that these graph classes occur naturally in power graphs and commuting graphs of finite groups. However, the higher strands of the Betti table, the regularity, and the Hilbert series were left open there. This omission is substantial, knowing only the linear strand does not determine the full homological behavior, does not locate the regularity corner, and does not answer when the edge ideal has a $2$-linear resolution.
	
	A second motivation comes from algebraic graphs attached to finite groups. Power graphs and commuting graphs have received sustained attention in both graph theory and algebra, see the surveys and foundational papers \cite{AbawajyKelarevChowdhury2013,ChakrabartyGhoshSen2009,KelarevQuinn2002,KumarSelvaganeshCameronChelvam2021}. In \cite{RatherWang2026}, homological invariants of edge ideals of power graphs of cyclic groups were investigated. That article established exact results for prime powers, for certain semiprime cases, and for cyclic groups with order equal to the product of three distinct primes. It also obtained inequalities for general values of $n$. Yet even there, the complete graded Betti table remained inaccessible in general, and the computations depended strongly on the special shape of the underlying power graph. Thus two parallel gaps were visible in the literature. First, the absence of full Betti tables for the graph families underlying many group graphs, and second, the lack of a structural mechanism explaining why apparently different algebraic graphs exhibit similar homological patterns.
	
	The present paper fills both gaps. We introduce two heterogeneous graph families
	$$
	\MM(a;\bk{b})=\overline{K}_a * \bigsqcup_{r=1}^m K_{b_r}
	\qquad\text{and}\qquad
	\CC(a;\bk{b})=K_a * \bigsqcup_{r=1}^m K_{b_r},
	$$
	where $a\geq 1$, $\bk{b}=(b_1,\dots,b_m)$ with each $b_r\geq 1$, and $*$ denotes graph join (joining each vertex of first graph to every vertex of second graph). When all $b_r$ are equal, these recover the multiple complete split-like graphs and clique-star graphs studied in \cite{Rather2024}. The main point is that the independence complexes of these families admit a transparent decomposition into a disjoint union of a simplex or a discrete complex with an iterated join of discrete complexes. This decomposition is sufficiently precise to make Hochster's formula completely explicit. In particular, every nonzero graded Betti number can be expressed either by a closed summation formula or, more compactly, by a coefficient extraction from certain elementary symmetric polynomials built from one-variable block polynomials. The resulting formulas simultaneously recover all previously known initial Betti numbers and produce every higher strand that was previously unavailable.
	
	The significance of these formulas is both conceptual and computational. Conceptually, they show that the entire Betti table is governed by two pieces of data only, the size $a$ of the distinguished part and the multiset of block sizes $\{b_1,\dots,b_m\}$. This immediately yields sharp formulas for regularity, projective dimension, Hilbert series, and extremal Betti numbers. It also gives a transparent criterion for $2$-linear resolution, as the edge ideal is $2$-linear precisely when there is at most one block of size at least $2$. Computationally, the formulas bypass repeated Gröbner or free-resolution computations in software systems such as Macaulay2 \cite{GraysonStillman2002}, instead, they reduce the problem to coefficient extraction from short generating functions. This leads to a dynamic programming algorithm with polynomial complexity in the total number of vertices and clique blocks.
	
	The real and technical significance of the problem is broader than the graph families themselves. Power graphs and commuting graphs encode algebraic generation and commutation, respectively, and therefore serve as combinatorial shadows of finite groups. Exact homological invariants of their edge ideals provide a refined measure of how these algebraic relations interact. For instance, the number and sizes of maximal cyclic or commuting pieces determine not only the shape of the graph but also the position of extremal Betti numbers. Such formulas are useful in symbolic computation, in the classification of graphs with low regularity, and in the systematic search for graph classes with prescribed homological behavior. They also offer a clean route from group-theoretic structure theorems to explicit free resolutions.
	
	Our main results may be summarized informally as follows. First, we determine the full graded Betti numbers of $\MM(a;\bk{b})$ and $\CC(a;\bk{b})$ for arbitrary block sizes. Second, we compute the Hilbert series of the corresponding Stanley--Reisner rings. Third, we show that the regularity of $R/I(G)$ equals $\max\{1,\nu\}$, where $\nu$ is the number of blocks with size at least $2$, and hence $\reg I(G)=1+\max\{1,\nu\}$. Fourth, we identify both the regularity corner and the last-column extremal corner, thereby determining the projective dimension. Finally, we apply the theory to pineapple graphs, power graphs of cyclic prime-power groups, power graphs of elementary abelian groups, and prime-power dihedral groups, recovering known results from \cite{Rather2024,RatherWang2026} and extending them substantially.
	
	A further goal of the paper is methodological. The formulas are not presented as isolated identities but as consequences of a unifying topological mechanism. The independence complex of our graph families is a disjoint union in which one component is an iterated join of discrete complexes. Since the reduced homology of such joins is easy to control, Hochster's formula becomes tractable in every bidegree. This perspective clarifies why equal-block and group-graph examples behave so regularly and why the top nonzero strand is governed by the number of nontrivial blocks.
	
	\medskip
	
	The paper is organized as: Section~\ref{sec:prelim} collects notation, recalls Hochster's formula, Fröberg's criterion, and the homology of joins, and isolates the precise gap left by the existing literature. Section~\ref{sec:decomp} develops the structural decomposition of the relevant independence complexes and derives explicit Hilbert series formulas. Section~\ref{sec:split} contains the first main family of new results, exact graded Betti numbers for generalized split--join graphs $\MM(a;\bk{b})$, including higher strands and equal-block specializations. Section~\ref{sec:clique} treats generalized clique-star graphs $\CC(a;\bk{b})$ in the same spirit and shows that the higher strands coincide with those of the split case while the linear strand changes in a controlled way. Section~\ref{sec:invariants} extracts regularity, $2$-linear criteria, extremal Betti numbers, projective dimension, and monotonicity consequences from the explicit Betti formulas. Section~\ref{sec:groups} applies the general theory to algebraic graphs coming from groups, including pineapple graphs, cyclic prime-power groups, elementary abelian groups, and prime-power dihedral groups. The paper ends with Section~\ref{sec:conclusion}, where we summarize the contributions, discuss limitations, and point to further problems.
	
	\section{Preliminaries and notation}\label{sec:prelim}
	
	Throughout the paper, $\kk$ denotes a field, all graphs are finite and simple, and for a graph $G$ we write $V(G)$ and $E(G)$ for its vertex and edge sets. For $W\subseteq V(G)$, the induced subgraph on $W$ is denoted by $G[W]$. We write $\overline{K}_a$ for the independent graph on $a$ vertices and $K_a$ for the complete graph on $a$ vertices. The symbol $\sqcup$ denotes disjoint union of graphs.

	\begin{definition}
		Let $G$ be a graph on $N$ vertices and let $R=\kk[x_1,\dots,x_N]$ be the polynomial ring indexed by the vertices of $G$. The edge ideal of $G$ is
		$$
		I(G)=(x_ix_j:{x_i,x_j}\in E(G))\subseteq R.
		$$
		We write
		$$
		\beta_{i,j}(G)=\beta_{i,j}(R/I(G))=\dim_{\kk}\Tor_i^R(\kk,R/I(G))_j.
		$$
		The Castelnuovo--Mumford regularity of $R/I(G)$ is defined by
		$$
		\reg(G):=\reg(R/I(G))=\max\{j-i:\beta_{i,j}(R/I(G))\neq 0\}.
		$$
		The projective dimension of $R/I(G)$ is defined by
		$$
		\pd(G):=\pd(R/I(G))=\max\{i:\beta_{i,j}(R/I(G))\neq 0\text{ for some }j\}.
		$$
		Equivalently, if
		$$
		0\to \bigoplus_j R(-j)^{\beta_{p,j}}\to \cdots \to \bigoplus_j R(-j)^{\beta_{1,j}}\to R\to R/I(G)\to 0
		$$
		is the minimal graded free resolution of $R/I(G)$, then $\reg(G)$ is the largest value of $j-i$ for which $\beta_{i,j}\neq 0$, and $\pd(G)$ is the length $p$ of the minimal graded free resolution.
	\end{definition}

	\begin{definition}
		A simplicial complex $\Delta$ on a vertex set $V$ is a collection of subsets of $V$ such that if $F \in \Delta$ and $G \subseteq F$, then $G \in \Delta$. A facet of $\Delta$ is a face $F \in \Delta$ that is maximal under inclusion, meaning there is no face $H \in \Delta$ with $F \subsetneq H$.	A facet is a maximal face of $\Delta$ under inclusion.
		
		The independence complex of a graph $G$ is the simplicial complex
		$$
		\ind(G)=\{F\subseteq V(G):F\text{ is an independent set in }G\}.
		$$
		If $\Delta$ is a simplicial complex and $W\subseteq V(\Delta)$, then $\Delta_W$ denotes the induced subcomplex on $W$.
	\end{definition}
	
	\begin{definition}
		For integers $a\geq 1$, $m\geq 1$, and $\bk{b}=(b_1,\dots,b_m)$ with $b_r\geq 1$, we define
		$$
		\MM(a;\bk{b})=\overline{K}_a * \bigsqcup_{r=1}^m K_{b_r}
		\qquad\text{and}\qquad
		\CC(a;\bk{b})=K_a * \bigsqcup_{r=1}^m K_{b_r}.
		$$
		When $b_1=\cdots=b_m=b$, these become the equal-block families studied in \cite{Rather2024}. In that notation,
		$ \MM(a;(b,\dots,b))=MCS^a_{b,m},$ and $
		\CC(a;(b,\dots,b))=S^a_{b,m}. $
	\end{definition}
	
	For later use, we set
	$ B=\sum_{r=1}^m b_r,$  $N=a+B,$ $\nu=\#\{r: b_r\geq 2\},$ and $\sigma=\sum_{b_r\geq 2} b_r.$ 	We also define
	$$
	M_s(\bk{b})=\binom{B}{s}-\sum_{r=1}^m \binom{b_r}{s}.
	$$
	So, $M_s(\bk{b})$ counts $s$-subsets of the union of the clique blocks that meet at least two distinct blocks.

	\medskip The next theorem is the basic bridge between graded Betti numbers and simplicial topology.
	
	\begin{theorem}[Hochster, \cite{Hochster1977,HerzogHibi2010}]\label{thm:Hochster}
		Let $\Delta$ be a simplicial complex on the vertex set $[N]$. Then for every $i,j\geq 0$,
		$$
		\beta_{i,j}(\kk[\Delta])=\sum_{\substack{W\subseteq [N]\\ |W|=j}}
		\dim_{\kk}\wt H_{j-i-1}(\Delta_W;\kk).
		$$
		In particular, if $\Delta=\ind(G)$, then $\beta_{i,j}(G)$ is computed by the reduced homology of induced subcomplexes of $\ind(G)$.
	\end{theorem}

	 \medskip The next criterion identifies the $2$-linear case and is the standard benchmark for edge ideals.
	
	\begin{theorem}[Fröberg, \cite{Froberg1990,Froberg2023}]\label{thm:Froberg}
		For a finite simple graph $G$, the edge ideal $I(G)$ has a $2$-linear resolution if and only if the complement graph $\overline{G}$ is chordal.
	\end{theorem}
	
	 \medskip The homology of joins turns the combinatorics of clique blocks into explicit higher strands.
	
	\begin{theorem}[Homology of joins, \cite{BrunsHerzog1998,HerzogHibi2010}]\label{thm:join-homology}
		Let $\Delta$ and $\Gamma$ be simplicial complexes on disjoint vertex sets. Then
		$$
		\wt H_q(\Delta * \Gamma;\kk)\cong \bigoplus_{p+\ell=q-1}\wt H_p(\Delta;\kk)\otimes_{\kk}\wt H_\ell(\Gamma;\kk).
		$$
	\end{theorem}

	The existing literature makes clear that the linear strand is only the beginning of the homological story. Equal-block formulas from \cite{Rather2024} show that certain split--join graphs are tractable in the first strand, but they do not reveal the higher syzygies, and hence they do not determine regularity or the full Betti table. Similarly, \cite{RatherWang2026} identifies interesting cyclic power-graph cases and proves sharp results in some special arithmetic regimes, but the method there remains tied to the specific subgroup pattern of the cyclic group under consideration. What is missing is a structural theorem covering arbitrary block sizes and all strands simultaneously. The purpose of the present paper is precisely to provide such a theorem and then to push it back to the algebraic graph classes that motivated the problem.
	
	\section{Homological decomposition and Hilbert series}\label{sec:decomp}
	
	In this section, we identify the exact simplicial structure behind the graph families $\MM(a;\bk{b})$ and $\CC(a;\bk{b})$. This is the key point on which all later formulas depend.
	
	Let $A$ denote the distinguished part of size $a$, and let $B_r$ be the vertex set of the $r$-th clique block $K_{b_r}$. Thus, we derive
	$$
	V(\MM(a;\bk{b}))=A\sqcup B_1\sqcup\cdots\sqcup B_m
	\qquad\text{and}\qquad
	V(\CC(a;\bk{b}))=A\sqcup B_1\sqcup\cdots\sqcup B_m.
	$$
	We also write $D_s$ for the discrete simplicial complex on $s$ vertices.
	
	 \medskip The next proposition shows that the independence complexes split into a very small number of topological pieces.
	
	\begin{proposition}\label{prop:decomp}
		The independence complexes of the two graph families satisfy
		$$
		\ind(\MM(a;\bk{b}))=\langle A\rangle \sqcup (D_{b_1} * \cdots * D_{b_m}),
		$$
		and
		$$
		\ind(\CC(a;\bk{b}))=D_a \sqcup (D_{b_1} * \cdots * D_{b_m}),
		$$
		where $\langle A\rangle$ denotes the simplex on the vertex set $A$.
	\end{proposition}
	
	\begin{proof}
		For $G=\MM(a;\bk{b})$, every vertex of $A$ is adjacent to every vertex of every block $B_r$, so no independent set can meet both $A$ and $\bigcup_r B_r$. Inside $A$ there are no edges, and hence every subset of $A$ is independent, thereby giving the simplex $\langle A\rangle$. Inside each block $B_r$ every pair of vertices is adjacent, so an independent set may contain at most one vertex from each block. Thus, across different blocks there are no edges, so one may choose independently from distinct blocks. This is exactly the iterated join $D_{b_1} * \cdots * D_{b_m}$. Hence the first formula follows.
		
		For $G=\CC(a;\bk{b})$, the same join restriction between $A$ and the blocks remains valid, but now the induced graph on $A$ is a clique, so independent sets in $A$ are only the empty set and the singletons. Therefore, the component contributed by $A$ is $D_a$, and the second formula follows.
	\end{proof}
	
	 \medskip The next lemma computes the entire reduced homology of the block-join component.
	
	\begin{lemma}\label{lem:block-homology}
		Let $\bk{w}=(w_1,\dots,w_t)$ with each $w_i\geq 1$, and let
		$ J(\bk{w})=D_{w_1} * \cdots * D_{w_t}. $
		Then the following hold.
		
		\begin{enumerate}[label=\textup{(\roman*)},noitemsep]
			\item If some $w_i=1$, then $J(\bk{w})$ is a cone, and hence contractible.
			\item If $w_i\geq 2$ for every $i$, then
			$$
			\wt H_q(J(\bk{w});\kk)\cong
			\begin{cases}
				\kk^{\prod_{i=1}^t (w_i-1)}, & q=t-1,\\
				0, & q\neq t-1.
			\end{cases}
			$$
		\end{enumerate}
	\end{lemma}
	
	\begin{proof}
		If some $w_i=1$, then $D_{w_i}$ is a point, and the join with a point is a cone. This proves (i).
		 For (ii), each $D_{w_i}$ has reduced homology concentrated in degree $0$, with
		$$
		\wt H_0(D_{w_i};\kk)\cong \kk^{w_i-1}.
		$$
		Applying Theorem~\ref{thm:join-homology} inductively, the only nonzero reduced homology of the join occurs in degree $t-1$, and the rank multiplies at each step. Hence, we have
		$$
		\wt H_{t-1}(J(\bk{w});\kk)\cong \bigotimes_{i=1}^t \wt H_0(D_{w_i};\kk)\cong \kk^{\prod_{i=1}^t (w_i-1)},
		$$
		while all other reduced homology groups vanish.
	\end{proof}
	
	 \medskip The next proposition reduces every induced subcomplex to one of a few explicit topological models.
	
	\begin{proposition}\label{prop:induced-models}
		Let $W\subseteq A\sqcup B_1\sqcup\cdots\sqcup B_m$, and let
		$ u=|W\cap A|, $ $ 	w_r=|W\cap B_r|, $ and $t=\#\{r:w_r>0\}. $ 	Then the induced subcomplexes are described as follows.
		
		\begin{enumerate}[label=\textup{(\roman*)},noitemsep]
			\item For $G=\MM(a;\bk{b})$, we have
			$$
			\ind(G)_W \cong
			\begin{cases}
				\langle W\cap A\rangle, & t=0,\\
				D_{w_r}, & u=0,\ t=1,\\
				J(w_{r_1},\dots,w_{r_t}), & u=0,\ t\geq 2,\\
				\langle W\cap A\rangle \sqcup D_{w_r}, & u>0,\ t=1,\\
				\langle W\cap A\rangle \sqcup J(w_{r_1},\dots,w_{r_t}), & u>0,\ t\geq 2.
			\end{cases}
			$$
			
			\item For $G=\CC(a;\bk{b})$, we have
			$$
			\ind(G)_W \cong
			\begin{cases}
				D_u, & t=0,\\
				D_{w_r}, & u=0,\ t=1,\\
				J(w_{r_1},\dots,w_{r_t}), & u=0,\ t\geq 2,\\
				D_u\sqcup D_{w_r}, & u>0,\ t=1,\\
				D_u\sqcup J(w_{r_1},\dots,w_{r_t}), & u>0,\ t\geq 2.
			\end{cases}
			$$
		\end{enumerate}
	\end{proposition}
	
	\begin{proof}
		This is an immediate induced-subcomplex version of Proposition~\ref{prop:decomp}. One simply restricts the components there to the vertices in $W$. For the split case the $A$-part remains a simplex whenever $u>0$. For the clique-star case it remains a discrete complex on $u$ vertices.
	\end{proof}
	
	 \medskip The next theorem computes the Hilbert series explicitly before any Betti-number calculation is attempted.
	
	\begin{theorem}\label{thm:Hilbert}
		Let
		$ F_{\MM}(x)=(1+x)^a+\prod_{r=1}^m (1+b_r x)-1 $
		and
		$ F_{\CC}(x)=1+ax+\prod_{r=1}^m (1+b_r x)-1. $
		Then
		$$
		\Hilb_{\kk[\ind(\MM(a;\bk{b}))]}(t)=F_{\MM}\Big(\tfrac{t}{1-t}\Big)
		=\frac{1}{(1-t)^a}+\prod_{r=1}^m\Big(1+\frac{b_r t}{1-t}\Big)-1,
		$$
		and
		$$
		\Hilb_{\kk[\ind(\CC(a;\bk{b}))]}(t)=F_{\CC}\Big(\tfrac{t}{1-t}\Big)
		=1+\frac{at}{1-t}+\prod_{r=1}^m\Big(1+\frac{b_r t}{1-t}\Big)-1.
		$$
	\end{theorem}
	
	\begin{proof}
		For a simplicial complex $\Delta$, we have
		$$
		\Hilb_{\kk[\Delta]}(t)=\sum_{F\in\Delta}\Big(\frac{t}{1-t}\Big)^{|F|}.
		$$
		Thus it is enough to determine the face enumerators. By Proposition~\ref{prop:decomp}, independent sets in $\MM(a;\bk{b})$ are either arbitrary subsets of $A$ or choices of at most one vertex from each block, but not both simultaneously. This fact gives us
		$$
		F_{\MM}(x)=\sum_{i=0}^a \binom{a}{i}x^i+\prod_{r=1}^m(1+b_r x)-1
		=(1+x)^a+\prod_{r=1}^m(1+b_r x)-1.
		$$
		Similarly, in $\CC(a;\bk{b})$ the independent sets in $A$ are only the empty set and singletons, and we have
		$$
		F_{\CC}(x)=1+ax+\prod_{r=1}^m(1+b_r x)-1.
		$$
		Now with $x=\tfrac{t}{1-t}$, we obtain the required series.
	\end{proof}
	
	 \medskip The next corollary shows that the heterogeneous formulas immediately collapse to the equal-block cases, with $b_r=b$ for all $r$ in Theorem~\ref{thm:Hilbert}.
	
	\begin{corollary}\label{cor:Hilbert-equal}
		If $b_1=\cdots=b_m=b$, then
		$$
		\Hilb_{\kk[\ind(\MM(a;(b,\dots,b)))]}(t)=\frac{1}{(1-t)^a}+\Big(1+\frac{bt}{1-t}\Big)^m-1,
		$$
		and
		$$
		\Hilb_{\kk[\ind(\CC(a;(b,\dots,b)))]}(t)=1+\frac{at}{1-t}+\Big(1+\frac{bt}{1-t}\Big)^m-1.
		$$
	\end{corollary}

	\medskip
	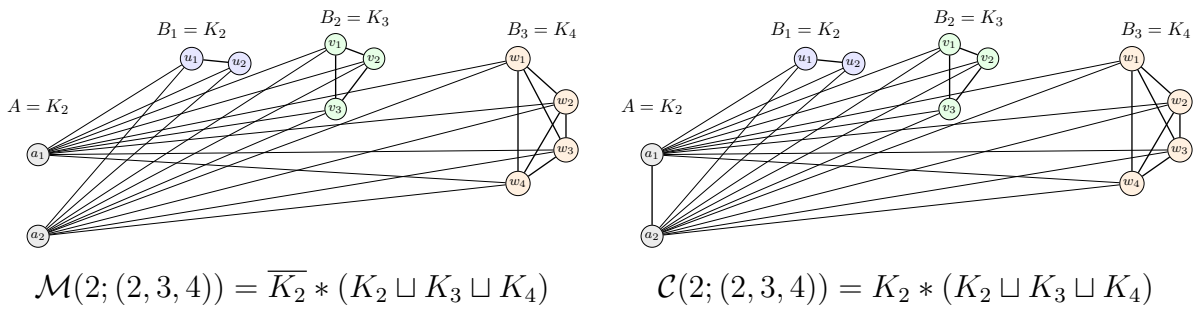
\begin{figure}[H]
		\centering
		
		\begin{minipage}{0.49\textwidth}
			\centering
			\resizebox{\textwidth}{!}{
				\begin{tikzpicture}[every node/.style={circle,draw,inner sep=1.2pt,font=\scriptsize}]
					\node[fill=gray!15] (a1) at (0,0.5) {$a_1$};
					\node[fill=gray!15] (a2) at (0,-1.2) {$a_2$};
					
					```
					\node[fill=blue!10] (u1) at (3.2,2.5) {$u_1$};
					\node[fill=blue!10] (u2) at (4.2,2.4) {$u_2$};
					
					\node[fill=green!10] (v1) at (6.2,2.8) {$v_1$};
					\node[fill=green!10] (v2) at (7.0,2.5) {$v_2$};
					\node[fill=green!10] (v3) at (6.2,1.45) {$v_3$};
					
					\node[fill=orange!12] (w1) at (10.0,2.5) {$w_1$};
					\node[fill=orange!12] (w2) at (11.0,1.6) {$w_2$};
					\node[fill=orange!12] (w3) at (11.0,0.6) {$w_3$};
					\node[fill=orange!12] (w4) at (10.0,-0.1) {$w_4$};
					
					\foreach \x in {u1,u2,v1,v2,v3,w1,w2,w3,w4}{
						\draw[thin] (a1)--(\x);
						\draw[thin] (a2)--(\x);
					}
					
					\draw[thick] (u1)--(u2);
					\draw[thick] (v1)--(v2)--(v3)--(v1);
					\draw[thick] (w1)--(w2)--(w3)--(w4)--(w1);
					\draw[thick] (w1)--(w3);
					\draw[thick] (w2)--(w4);
					
					\node[draw=none,fill=none,font=\small] at (0,1.5) {$A=K_2$};
					\node[draw=none,fill=none,font=\small] at (3.2,3.1) {$B_1=K_2$};
					\node[draw=none,fill=none,font=\small] at (6.6,3.35) {$B_2=K_3$};
					\node[draw=none,fill=none,font=\small] at (10.5,3.1) {$B_3=K_4$};
					
				\end{tikzpicture}
			}
			
			\vspace{1.5mm}
			$\MM(2;(2,3,4))=\overline{K_2}*(K_2\sqcup K_3\sqcup K_4)$
		\end{minipage}
		\hfill
		\begin{minipage}{0.49\textwidth}
			\centering
			\resizebox{\textwidth}{!}{
				\begin{tikzpicture}[every node/.style={circle,draw,inner sep=1.2pt,font=\scriptsize}]
					\node[fill=gray!15] (a1) at (0,0.5) {$a_1$};
					\node[fill=gray!15] (a2) at (0,-1.2) {$a_2$};

					\node[fill=blue!10] (u1) at (3.2,2.5) {$u_1$};
					\node[fill=blue!10] (u2) at (4.2,2.4) {$u_2$};
					
					\node[fill=green!10] (v1) at (6.2,2.8) {$v_1$};
					\node[fill=green!10] (v2) at (7.0,2.5) {$v_2$};
					\node[fill=green!10] (v3) at (6.2,1.45) {$v_3$};
					
					\node[fill=orange!12] (w1) at (10.0,2.5) {$w_1$};
					\node[fill=orange!12] (w2) at (11.0,1.6) {$w_2$};
					\node[fill=orange!12] (w3) at (11.0,0.6) {$w_3$};
					\node[fill=orange!12] (w4) at (10.0,-0.1) {$w_4$};
					
					\draw[thick] (a1)--(a2);
					
					\foreach \x in {u1,u2,v1,v2,v3,w1,w2,w3,w4}{
						\draw[thin] (a1)--(\x);
						\draw[thin] (a2)--(\x);
					}
					
					\draw[thick] (u1)--(u2);
					\draw[thick] (v1)--(v2)--(v3)--(v1);
					\draw[thick] (w1)--(w2)--(w3)--(w4)--(w1);
					\draw[thick] (w1)--(w3);
					\draw[thick] (w2)--(w4);
					
					\node[draw=none,fill=none,font=\small] at (0,1.5) {$A=K_2$};
					\node[draw=none,fill=none,font=\small] at (3.2,3.1) {$B_1=K_2$};
					\node[draw=none,fill=none,font=\small] at (6.6,3.35) {$B_2=K_3$};
					\node[draw=none,fill=none,font=\small] at (10.5,3.1) {$B_3=K_4$}; 
					
				\end{tikzpicture}
			}
			
			\vspace{1.5mm}
			$\CC(2;(2,3,4))=K_2*(K_2\sqcup K_3\sqcup K_4)$
		\end{minipage}
		
		\caption{Graphs $\MM(2;(2,3,4))$ and $\CC(2;(2,3,4))$.}
		\label{fig:mm-cc-234-clean}
	\end{figure}
	
	\begin{example}\label{ex:hilbert-example}
		 For the graph $\MM(2;(2,3,4))$, the set $A$ is independent. Hence every subset of $A$ is an independent set, and this contributes $ 1+2x+x^2=(1+x)^2 $ to the face polynomial. On the other hand, from the three clique blocks one may choose at most one vertex from each block, so their contribution is
		$ (1+2x)(1+3x)(1+4x). $
		The empty set appears in both pieces, so it must be counted only once. Therefore, we have
		$$
		F_{\MM}(x)=(1+x)^2+(1+2x)(1+3x)(1+4x)-1=1+11x+27x^2+24x^3.
		$$
		With $x=\tfrac{t}{1-t}$, we obtain Hilbert series as
		$$
		\Hilb_{\kk[\ind(\MM(2;(2,3,4)))]}(t)
		= 1+11\frac{t}{1-t}+27\frac{t^2}{(1-t)^2}+24\frac{t^3}{(1-t)^3}=\frac{1+8t+8t^2+7t^3}{(1-t)^3}.
		$$
		With the same idea for $\CC(2;(2,3,4))$, we have 
		$$
		F_{\CC}(x)=1+11x+26x^2+24x^3.
		$$
		With $x=\tfrac{t}{1-t}$, we obtain 
		$$
		\Hilb_{\kk[\ind(\CC(2;(2,3,4)))]}(t)
		= \frac{1+8t+7t^2+8t^3}{(1-t)^3}.
		$$
		
		The difference between the two families is now transparent. In the split--join case $\MM(2;(2,3,4))$, the pair ${a_1,a_2}$ is an independent set, so an extra $x^2$ term appears in the face polynomial. In the clique--star case $\CC(2;(2,3,4))$, this pair is forbidden, and that single missing $2$-face changes the numerator of the Hilbert series. Figure \ref{fig:decomp-detailed} illustrates the face decomposition of both graphs.
	\end{example}
	
	\begin{figure}[H]
		\centering
		\begin{tikzpicture}[scale=1, every node/.style={font=\small}, >=Latex]
			\node[draw, rounded corners, minimum width=3.1cm, minimum height=1.15cm, fill=gray!12] (L1) at (-.1,-.2) {};
			\node at (L1.center) {$\langle A\rangle$};
			\node[below=0.1cm of L1] {$A=\{a_1,a_2\}$ gives a $1$-simplex};
			
			```
			\node[draw, rounded corners, minimum width=5.2cm, minimum height=1.15cm, fill=blue!8] (R1) at (7.2,-.2) {};
			\node at (R1.center) {$D_2 * D_3 * D_4$};
			\node[below=0.1cm of R1] {choose at most one vertex from each block};
			
			\draw[thick,-Latex] (1.7,0) -- (4.6,0);
			\node[above=0.18cm] at (3.15,0) {$\sqcup$};
			
			\node[above=0.85cm of $(L1)!0.5!(R1)$] {$\ind(\MM(2;(2,3,4)))=\langle A\rangle \sqcup (D_2 * D_3 * D_4)$};
			
			\node[draw, rounded corners, minimum width=3.1cm, minimum height=1.15cm, fill=gray!12] (L2) at (0,-3.2) {};
			\node at (L2.center) {$D_A$};
			\node[below=0.1cm of L2] {$A=\{a_1,a_2\}$ contributes two $0$-simplex};
			
			\node[draw, rounded corners, minimum width=5.2cm, minimum height=1.15cm, fill=green!10] (R2) at (7.2,-3.2) {};
			\node at (R2.center) {$D_2 * D_3 * D_4$};
			\node[below=0.1cm of R2] {the block contribution is unchanged};
			
			\draw[thick,-Latex] (1.7,-3.1) -- (4.6,-3.1);
			\node[above=0.18cm] at (3.15,-3.1) {$\sqcup$};
			
			\node[above=0.85cm of $(L2)!0.5!(R2)$] {$\ind(\CC(2;(2,3,4)))=D_A \sqcup (D_2 * D_3 * D_4)$};
			
			\begin{scope}[shift={(-0.8,0)}]
				\filldraw[black] (0.3,0.1) circle (1.1pt);
				\filldraw[black] (1.1,0.1) circle (1.1pt);
				\draw (0.3,0.1)--(1.1,0.1);
			\end{scope}
			
			\begin{scope}[shift={(-0.8,-3.1)}]
				\filldraw[black] (0.3,0.1) circle (1.1pt);
				\filldraw[black] (1.1,0.1) circle (1.1pt);
			\end{scope}
			```
			
		\end{tikzpicture}
		\caption{A more explicit picture of the two independence complexes.}
		\label{fig:decomp-detailed}
	\end{figure}
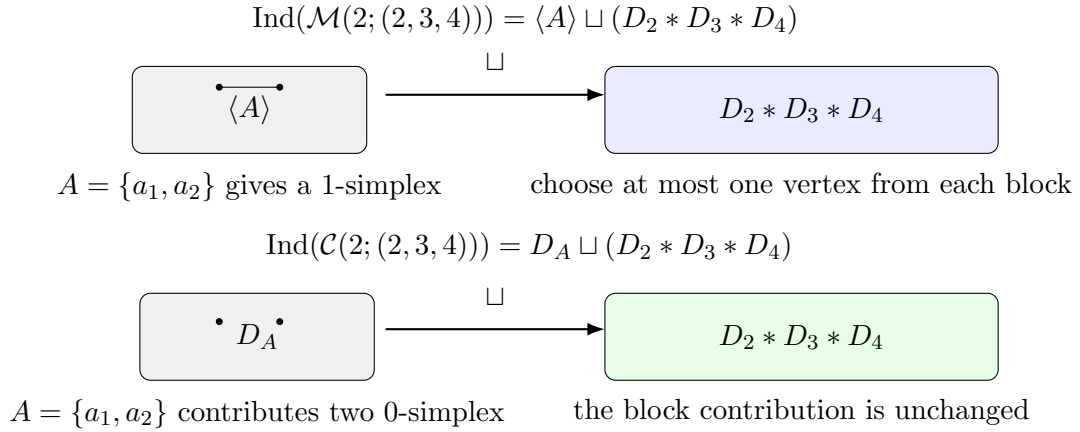

	The decomposition in Proposition~\ref{prop:decomp} is the point where the present paper departs from \cite{Rather2024,RatherWang2026}. Earlier work handled equal-block families and special group graphs largely case by case. Here the entire problem is reorganized around a topological decomposition valid for arbitrary block sizes, and Theorem~\ref{thm:Hilbert} already gives Hilbert series formulas that were not available in that generality.
	
	\section{Exact graded Betti numbers for generalized split--join graphs}\label{sec:split}
	
	We now turn to the family
	$ \MM(a;\bk{b})=\overline{K}_a * \bigsqcup_{r=1}^m K_{b_r}. $
	The equal-block case was treated only on the initial strand in \cite{Rather2024}. We derive the full graded Betti table.
	 For convenience, define the block polynomials
	$$
	P_r(z)=\sum_{q=2}^{b_r}(q-1)\binom{b_r}{q}z^q,
	\qquad 1\leq r\leq m,
	$$
	and for $t\geq 1$, let
	$$
	\EE_t(z;\bk{b})=e_t(P_1(z),\dots,P_m(z)),
	$$
	which is the $t$-th elementary symmetric polynomial in the polynomials $P_1(z),\dots,P_m(z)$.
	
	 \medskip The next theorem gives the complete linear strand of $\MM(a;\bk{b})$.
	
	\begin{theorem}\label{thm:split-linear}
		For $j\geq 2$,
		$$
		\beta_{j-1,j}(\MM(a;\bk{b}))=
		\sum_{r=1}^m (j-1)\binom{b_r}{j}
		+
		\sum_{u=1}^{\min{a,j-1}}
		\binom{a}{u}
		\sum_{r=1}^m (j-u)\binom{b_r}{j-u}
		+
		\sum_{u=1}^{\min{a,j-2}}
		\binom{a}{u}M_{j-u}(\bk{b}).
		$$
	\end{theorem}
	
	\begin{proof}
		For $G=\MM(a;\bk{b})$, Hochster's formula implies
		$$
		\beta_{j-1,j}(G)=\sum_{\substack{W\subseteq V(G)\\ |W|=j}}
		\dim_{\kk}\wt H_0(\ind(G)_W;\kk).
		$$
		We therefore classify all $j$-subsets $W$ for which $\ind(G)_W$ is disconnected, and consider the following cases.
		
		\smallskip
		
		\noindent\textbf{Case 1:} If $W\subseteq B_r$ for some fixed $r$, then $\ind(G)_W\cong D_j$, so we have
		$$
		\dim_{\kk}\wt H_0(\ind(G)_W;\kk)=j-1.
		$$
		There are $\binom{b_r}{j}$ such choices inside $B_r$, and hence the total contribution of this case is
		$$
		\sum_{r=1}^m (j-1)\binom{b_r}{j}.
		$$
		
		\smallskip
		
		\noindent\textbf{Case 2:} If $W$ meets $A$ and exactly one block $B_r$, then for $u=|W\cap A|$ with $1\leq u\leq j-1$, so we get $|W\cap B_r|=j-u$. By Proposition~\ref{prop:induced-models}, we have
		$$
		\ind(G)_W\cong \langle W\cap A\rangle \sqcup D_{j-u}.
		$$
		It has $(j-u)+1$ connected components, and hence we have
		$$
		\dim_{\kk}\wt H_0(\ind(G)_W;\kk)=j-u.
		$$
		For fixed $u$, the number of such subsets is $\binom{a}{u}\binom{b_r}{j-u}$, so the total contribution is
		$$
		\sum_{u=1}^{\min{a,j-1}}
		\binom{a}{u}
		\sum_{r=1}^m (j-u)\binom{b_r}{j-u}.
		$$
		
		\smallskip
		
		\noindent\textbf{Case 3:} If $W$ meets $A$ and at least two distinct blocks. Again For $u=|W\cap A|$, with $1\leq u\leq j-2$, and by Proposition~\ref{prop:induced-models}, we have
		$$
		\ind(G)_W\cong \langle W\cap A\rangle \sqcup J(\bk{w}),
		$$
		where $J(\bk{w})$ is connected, since it is the join of at least two nonempty complexes. Hence, $\ind(G)_W$ has exactly two connected components, so we obtain
		$$
		\dim_{\kk}\wt H_0(\ind(G)_W;\kk)=1.
		$$
		For fixed $u$, the number of ways to choose the remaining $j-u$ vertices in at least two blocks is precisely $M_{j-u}(\bk{b})$. SO, in this  case, the total  contribution is
		$$
		\sum_{u=1}^{\min{a,j-2}}
		\binom{a}{u}M_{j-u}(\bk{b}).
		$$
		
		\smallskip
		 No other subset contributes to $\beta_{j-1,j}(\MM(a;\bk{b}))$. Indeed, if $W\subseteq A$, then $\ind(G)_W$ is a simplex, if $W$ is contained in at least two blocks and avoids $A$, then $\ind(G)_W$ is connected, and contributes null. Now, summing the cases, we obtain the stated formula for $\beta_{j-1,j}(\MM(a;\bk{b})).$
	\end{proof}
	
	 \medskip The next theorem gives every higher strand in one formula.
	
	\begin{theorem}\label{thm:split-higher}
		Let $t\geq 2$. Then for every $j\geq 2t$,
		$$
		\beta_{j-t,j}(\MM(a;\bk{b}))
		= \sum_{u=0}^{\min{a,j-2t}}
		\binom{a}{u}[z^{j-u}]\EE_t(z;\bk{b}),
		$$
		where $[z^s]$ denotes coefficient extraction, or equivalently,
		$$
		\beta_{j-t,j}(\MM(a;\bk{b}))
		= \sum_{u=0}^{\min{a,j-2t}}
		\binom{a}{u}
		\sum_{\substack{1\leq r_1<\cdots<r_t\leq m\\
				q_1+\cdots+q_t=j-u\\
				2\leq q_i\leq b_{r_i}}}~
		\prod_{i=1}^t (q_i-1)\binom{b_{r_i}}{q_i}.
		$$
	\end{theorem}
	
	\begin{proof}
		For $t\geq 2$ and $j\geq 2t$ with $G=\MM(a;\bk{b})$, Hochster's formula implies that
		$$
		\beta_{j-t,j}(G)=\sum_{\substack{W\subseteq V(G)\\ |W|=j}}
		\dim_{\kk}\wt H_{t-1}(\ind(G)_W;\kk).
		$$
		By Proposition~\ref{prop:induced-models}, the only subsets $W$ that can contribute to $\wt H_{t-1}$ are those for which the block part of $W$ meets exactly $t$ blocks, each in at least $2$ vertices. Any selected block of size $1$ would make the relevant join a cone by Lemma~\ref{lem:block-homology}(i), and hence no contribution. Thus, for non zero contribution to $ \beta_{j-t,j}(\MM(a;\bk{b})),$ we choose: (i)  $u$ vertices from $A$, where $0\leq u\leq a$, (ii)   $t$ distinct blocks $B_{r_1},\dots,B_{r_t}$, (iii)   integers $q_i\geq 2$ with $q_1+\cdots+q_t=j-u$, and (iv)   $q_i$ vertices from each selected block $B_{r_i}$. 		
		For such a choice, Proposition~\ref{prop:induced-models} and Lemma~\ref{lem:block-homology}(ii) implies that
		$$
		\dim_{\kk}\wt H_{t-1}(\ind(G)_W;\kk)=\prod_{i=1}^t (q_i-1).
		$$
		The number of such subsets realizing these numbers is
		$$
		\binom{a}{u}\prod_{i=1}^t \binom{b_{r_i}}{q_i}.
		$$
		Therefore, we obtain 
		$$
		\beta_{j-t,j}(G)=
		\sum_{u=0}^{\min{a,j-2t}}
		\binom{a}{u}
		\sum_{\substack{1\leq r_1<\cdots<r_t\leq m\\
				q_1+\cdots+q_t=j-u\\
				2\leq q_i\leq b_{r_i}}}~
		\prod_{i=1}^t (q_i-1)\binom{b_{r_i}}{q_i},
		$$
		which is exactly the second displayed formula. The coefficient form follows, as selecting one monomial from each of $P_{r_1}(z),\dots,P_{r_t}(z)$ and summing over all $t$-subsets of blocks produces the elementary symmetric polynomial $\EE_t(z;\bk{b})$.
	\end{proof}
	
	 \medskip The next corollary packages all higher strands into a single generating function.
	
	\begin{corollary}\label{cor:split-gen}
		For each fixed $t\geq 2$,
		$$
		\sum_{j\geq 0}\beta_{j-t,j}(\MM(a;\bk{b}))\lambda^j
		=(1+\lambda)^a\EE_t(\lambda;\bk{b}).
		$$
	\end{corollary}
	
	\begin{proof}
		Multiply the coefficient formula in Theorem~\ref{thm:split-higher} by $\lambda^j$ and sum over $j$. The factor $(1+\lambda)^a$ records the choice of vertices from $A$, while $\EE_t(\lambda;\bk{b})$ records the selected blocks and their internal multiplicities. Now, with these facts, the formula follows.
	\end{proof}
	
	 \medskip The next corollary recovers the equal-block formulas from \cite{Rather2024} and extends them to every strand.
	
	\begin{corollary}\label{cor:split-equal}
		Assume $b_1=\cdots=b_m=b$. Then for $t\geq 2$,
		$$
		\beta_{j-t,j}(\MM(a;(b,\dots,b)))
		= \binom{m}{t}
		\sum_{u=0}^{\min{a,j-2t}}
		\binom{a}{u}[z^{j-u}]P_b(z)^t,
		$$
		where
		$$
		P_b(z)=\sum_{q=2}^{b}(q-1)\binom{b}{q}z^q.
		$$
		Moreover, the linear strand is given by Theorem~\ref{thm:split-linear} with all $b_r=b$.
	\end{corollary}
	
	\begin{proof}
		If all blocks have the same size $b$, then every $t$-fold elementary symmetric polynomial in $P_1(z),\dots,P_m(z)$ equals $\binom{m}{t}P_b(z)^t$. The linear strand is immediate from Theorem~\ref{thm:split-linear}.
	\end{proof}

	\begin{example}\label{ex:split-example}
		For  $G=\MM(2;(2,3,4))$, we have  
		$$
		M_s(\bk{b})=\binom{9}{s}-\binom{2}{s}-\binom{3}{s}-\binom{4}{s}.
		$$
		As $a=2$, the linear-strand formula becomes
		$$
		\beta_{j-1,j}(G)
		= \sum_{r=1}^3 (j-1)\binom{b_r}{j}
		+
		\sum_{u=1}^{\min{2,j-1}}
		\binom{2}{u}\sum_{r=1}^3 (j-u)\binom{b_r}{j-u}
		+
		\sum_{u=1}^{\min{2,j-2}}
		\binom{2}{u}M_{j-u}.
		$$
		 For $j=2$,
		$$
		\beta_{1,2}(G)
		= \sum_{r=1}^3 \binom{b_r}{2}
		+
		2\sum_{r=1}^3 \binom{b_r}{1}.
		 = \left(\binom{2}{2}+\binom{3}{2}+\binom{4}{2}\right)
		+
		2\left(\binom{2}{1}+\binom{3}{1}+\binom{4}{1}\right)
		 =28.
		$$
		For $j=3$,
		$$
		\beta_{2,3}(G)
		= \sum_{r=1}^3 2\binom{b_r}{3}
		+
		2\sum_{r=1}^3 2\binom{b_r}{2}
		+
		\sum_{r=1}^3 \binom{b_r}{1}
		+
		2M_2=111.
		$$
		
		\medskip
		\noindent
		For $j=4$,
		$$
		\beta_{3,4}(G)
		= \sum_{r=1}^3 3\binom{b_r}{4}
		+
		2\sum_{r=1}^3 3\binom{b_r}{3}
		+
		\sum_{r=1}^3 2\binom{b_r}{2}
		+
		2M_3+M_2=237.
		$$
		
		\medskip
		\noindent
		Proceeding in the same way for $j=5,6,\dots,11$, we obtain the full linear strand as
		$$
		\begin{aligned}
			\beta_{1,2}(G)&=28, &
			\beta_{2,3}(G)&=111, &
			\beta_{3,4}(G)&=237, \\
			\beta_{4,5}(G)&=352, &
			\beta_{5,6}(G)&=381, &
			\beta_{6,7}(G)&=294, \\
			\beta_{7,8}(G)&=156, &
			\beta_{8,9}(G)&=54, &
			\beta_{9,10}(G)&=11, \\
			\beta_{10,11}(G)&=1. &&
		\end{aligned}
		$$
		For higher strands, the block polynomials are
		$$
		P_1(z)=\sum_{q=2}^{2}(q-1)\binom{2}{q}z^q=z^2,
		$$
		$$
		P_2(z)=\sum_{q=2}^{3}(q-1)\binom{3}{q}z^q=3z^2+2z^3,
		$$
		$$
		P_3(z)=\sum_{q=2}^{4}(q-1)\binom{4}{q}z^q=6z^2+8z^3+3z^4.
		$$
		 For the second strand,  we need
		$ \EE_2(z;\bk{b})=P_1P_2+P_1P_3+P_2P_3. $
		By direct computations, we have
		$$
		P_1P_2=z^2(3z^2+2z^3)=3z^4+2z^5,
		$$
		$$
		P_1P_3=z^2(6z^2+8z^3+3z^4)=6z^4+8z^5+3z^6,
		$$
		$$
		P_2P_3=(3z^2+2z^3)(6z^2+8z^3+3z^4)=18z^4+36z^5+25z^6+6z^7.
		$$
		Therefore, we obtain 
		$ \EE_2(z;\bk{b})=27z^4+46z^5+28z^6+6z^7. $
		Since $a=2$, Theorem~\ref{thm:split-higher} implies that 
		$ \beta_{j-2,j}(G)=[z^j](1+z)^2\EE_2(z;\bk{b}). $
		Now
		\begin{align*}
		 (1+z)^2\EE_2(z;\bk{b})
		=& (1+2z+z^2)(27z^4+46z^5+28z^6+6z^7),\\
		=& 27z^4+100z^5+147z^6+108z^7+40z^8+6z^9. 
		\end{align*}
		So, the second strand is
		$$
		\beta_{2,4}(G)=27,
		\beta_{3,5}(G)=100,
		\beta_{4,6}(G)=147, 
		\beta_{5,7}(G)=108,
		\beta_{6,8}(G)=40,
		\beta_{7,9}(G)=6.
		$$
		
		\medskip
		\noindent
		For the third strand,  we have
		$ \EE_3(z;\bk{b})=P_1P_2P_3. $
		With 
		$ P_1P_2=z^2(3z^2+2z^3)=3z^4+2z^5, $
		we get
		$$
		\EE_3(z;\bk{b})
		= (3z^4+2z^5)(6z^2+8z^3+3z^4)
		=
		18z^6+36z^7+25z^8+6z^9.
		$$
		For $a=2$,
		$$
		\beta_{j-3,j}(G)=[z^j](1+z)^2\EE_3(z;\bk{b}).
		$$
		Thus, we have
		\begin{align*}
		(1+z)^2\EE_3(z;\bk{b})
		=& (1+2z+z^2)(18z^6+36z^7+25z^8+6z^9),\\
		=&18z^6+72z^7+115z^8+92z^9+37z^{10}+6z^{11}.
		\end{align*}
		So the third strand is
		$$
		\beta_{3,6}(G)=18,
		\beta_{4,7}(G)=72,
		\beta_{5,8}(G)=115, 
		\beta_{6,9}(G)=92,
		\beta_{7,10}(G)=37,
		\beta_{8,11}(G)=6.
		$$
		
		\medskip
		\noindent
		Putting everything together, the complete list of nonzero graded Betti numbers is
		$$
		\begin{aligned}
			&\beta_{1,2}=28,\ \beta_{2,3}=111,\ \beta_{3,4}=237,\ \beta_{4,5}=352,\ \beta_{5,6}=381,\ \beta_{6,7}=294,\\
			&\beta_{7,8}=156,\ \beta_{8,9}=54,\ \beta_{9,10}=11,\ \beta_{10,11}=1,\\
			&\beta_{2,4}=27,\ \beta_{3,5}=100,\ \beta_{4,6}=147,\ \beta_{5,7}=108,\ \beta_{6,8}=40,\ \beta_{7,9}=6,\\
			&\beta_{3,6}=18,\ \beta_{4,7}=72,\ \beta_{5,8}=115,\ \beta_{6,9}=92,\ \beta_{7,10}=37,\ \beta_{8,11}=6.
		\end{aligned}
		$$
		 A small remark is worth recording here. The graph $G=\MM(2;(2,3,4))$ has exactly three nontrivial clique blocks, so the resolution has precisely three nonzero strands, the linear strand, the strand with $j-i=2$, and the strand with $j-i=3$. In particular, there are no Betti numbers with $j-i\geq 4$. This is exactly the kind of higher-strand behavior that does not show up if one only looks at the initial formulas in the equal-block setting.
	\end{example}

	Table \ref{tab:split-betti} gives the Betti numbers of $\MM(2;(2,3,4))$, and their interpretation.
	\begin{table}[H]
		\centering
		\caption{Selected Betti numbers of $G=\MM(2;(2,3,4))$ from Example~\ref{ex:split-example}. The third column is the direct coefficient extraction from Theorems~\ref{thm:split-linear} and \ref{thm:split-higher}.}
		\label{tab:split-betti}
		\begin{tabular}{@{}cccc@{}}
			\toprule
			Bidegree $(i,j)$ & Formula source & Value & Interpretation \\ \midrule
			$(1,2)$ & linear strand & $28$ & edge generators in the first strand \\
			$(2,3)$ & linear strand & $111$ & next linear syzygies \\
			$(2,4)$ & $t=2$ strand & $27$ & first genuinely higher strand \\
			$(3,6)$ & $t=3$ strand & $18$ & top regularity strand entry \\
			$(8,11)$ & $t=3$ strand & $6$ & near-terminal higher-strand entry \\
			$(10,11)$ & linear strand & $1$ & last-column extremal entry \\ \bottomrule
		\end{tabular}
	\end{table}
	
	 The formulas in Theorems~\ref{thm:split-linear} and \ref{thm:split-higher} are new even in the equal-block case. The work \cite{Rather2024} determined only the initial Betti numbers of $MCS^a_{b,m}$, whereas Corollary~\ref{cor:split-equal} produces every higher strand and therefore the entire Betti table. The heterogeneous version with arbitrary $b_1,\dots,b_m$ does not appear in the earlier literature.
	
	\section{Exact graded Betti numbers for generalized clique-star graphs}\label{sec:clique}
	We now consider 	$ \CC(a;\bk{b})=K_a * \bigsqcup_{r=1}^m K_{b_r}. $
	The topology of the block part is unchanged from Section~\ref{sec:split}, but the $A$-part is now discrete. This changes the linear strand while leaving the higher strands intact.
	
	 \medskip The next theorem shows exactly how the linear strand changes when the independent part becomes a clique.
	
	\begin{theorem}\label{thm:clique-linear}
		For $j\geq 2$,
		\begin{align*}
		 \beta_{j-1,j}(\CC(a;\bk{b}))=&
		(j-1)\binom{a}{j}
		+\sum_{r=1}^m (j-1)\binom{b_r}{j}
		+\sum_{u=1}^{\min{a,j-1}}
		(j-1)\binom{a}{u}\sum_{r=1}^m \binom{b_r}{j-u}\\
		&+\sum_{u=1}^{\min{a,j-2}}
		u\binom{a}{u}M_{j-u}(\bk{b}). 
		\end{align*}
	\end{theorem}
	
	\begin{proof}
		For $G=\CC(a;\bk{b})$, Hochster's formula gives
		$$
		\beta_{j-1,j}(G)=\sum_{\substack{W\subseteq V(G)\\ |W|=j}}
		\dim_{\kk}\wt H_0(\ind(G)_W;\kk).
		$$
		For non zero contribution to $\beta_{j-1,j}(G) $, we classify the contributing subsets.
		
		\smallskip
		
		\noindent\textbf{Case 1:} If $W\subseteq A$, then $\ind(G)_W\cong D_j$, so the contribution is $(j-1)\binom{a}{j}$.
		
		\smallskip
		
		\noindent\textbf{Case 2:} If $W\subseteq B_r$ for some $r$, then $\ind(G)_W\cong D_j$, contributing
		$$
		\sum_{r=1}^m (j-1)\binom{b_r}{j}.
		$$
		
		\smallskip
		
		\noindent\textbf{Case 3:} If $W$ meets $A$ and exactly one block $B_r$, then for $u=|W\cap A|$ and $j-u=|W\cap B_r|$, with $1\leq u\leq j-1$, we have
		$$
		\ind(G)_W\cong D_u\sqcup D_{j-u}.
		$$
		So, the number of connected components is $u+(j-u)=j$, and hence
		$$
		\dim_{\kk}\wt H_0(\ind(G)_W;\kk)=j-1.
		$$
		The contribution in this case is
		$$
		\sum_{u=1}^{\min{a,j-1}}
		(j-1)\binom{a}{u}\sum_{r=1}^m \binom{b_r}{j-u}.
		$$
		
		\smallskip
		
		\noindent\textbf{Case 4:} If $W$ meets $A$ and at least two distinct blocks, then for $u=|W\cap A|$,  Proposition~\ref{prop:induced-models} implies that
		$$
		\ind(G)_W\cong D_u\sqcup J(\bk{w}),
		$$
		where $J(\bk{w})$ is connected. Therefore the total number of connected components is $u+1$, and hence
		$$
		\dim_{\kk}\wt H_0(\ind(G)_W;\kk)=u.
		$$
		For fixed $u$, the number of such subsets is $\binom{a}{u}M_{j-u}(\bk{b})$, so this case contributes
		$$
		\sum_{u=1}^{\min{a,j-2}}
		u\binom{a}{u}M_{j-u}(\bk{b}).
		$$
		
		\smallskip
		 The remaining subsets cases contribute nothing as $\ind(G)_W$ is a connected component. Now summing the four cases,  the result follows.
	\end{proof}
	
	 \medskip The next theorem shows that all higher strands are unchanged from the split--join case.
	
	\begin{theorem}\label{thm:clique-higher}
		Let $t\geq 2$. Then for every $j\geq 2t$,
		$$
		\beta_{j-t,j}(\CC(a;\bk{b}))
		= \sum_{u=0}^{\min{a,j-2t}}
		\binom{a}{u}[z^{j-u}]\EE_t(z;\bk{b}).
		$$
		or equivalently,
		$$
		\beta_{j-t,j}(\CC(a;\bk{b}))
		= \sum_{u=0}^{\min{a,j-2t}}
		\binom{a}{u}
		\sum_{\substack{1\leq r_1<\cdots<r_t\leq m\\
				q_1+\cdots+q_t=j-u\\
				2\leq q_i\leq b_{r_i}}}\,
		\prod_{i=1}^t (q_i-1)\binom{b_{r_i}}{q_i}.
		$$
	\end{theorem}
	
	\begin{proof}
		By Proposition~\ref{prop:induced-models}, the only difference between $\ind(\CC(a;\bk{b}))_W$ and $\ind(\MM(a;\bk{b}))_W$ is that the simplex on $A$ is replaced by a discrete set. But for $t\geq 2$ we are computing $\wt H_{t-1}$ with $t-1\geq 1$. The discrete part on $A$ has no positive-dimensional reduced homology, so the only contribution again comes from the join component of the selected blocks. Exactly the same block counting as in Theorem~\ref{thm:split-higher} therefore applies.
	\end{proof}
	
	 \medskip The next corollary packages the higher strands of the clique-star family into the same generating function as before.
	
	\begin{corollary}\label{cor:clique-gen}
		For each fixed $t\geq 2$,
		$$
		\sum_{j\geq 0}\beta_{j-t,j}(\CC(a;\bk{b}))\lambda^j
		=(1+\lambda)^a\EE_t(\lambda;\bk{b}).
		$$
	\end{corollary}
	
	\begin{proof}
		The proof is identical to that of Corollary~\ref{cor:split-gen}, using Theorem~\ref{thm:clique-higher}.
	\end{proof}
	
	 \medskip The next corollary as an application of Theorems~\ref{thm:clique-linear} and \ref{thm:clique-higher}, recovers the  equal-block clique-star formulas from \cite{Rather2024} to the complete graded Betti table.
	
	\begin{corollary}\label{cor:clique-equal}
		Assume $b_1=\cdots=b_m=b$. Then for $t\geq 2$,
		$$
		\beta_{j-t,j}(\CC(a;(b,\dots,b)))
		= \binom{m}{t}
		\sum_{u=0}^{\min{a,j-2t}}
		\binom{a}{u}[z^{j-u}]P_b(z)^t,
		$$
		with
		$$
		P_b(z)=\sum_{q=2}^{b}(q-1)\binom{b}{q}z^q.
		$$
		The linear strand is given by Theorem~\ref{thm:clique-linear} with all $b_r=b$.
	\end{corollary}

	\begin{example}\label{ex:clique-example}
		Let $G=\CC(2;(2,3,4))$. Then $N =11,$ and Theorems~\ref{thm:clique-linear} and \ref{thm:clique-higher}, we have  $
		B=b_1+b_2+b_3=9, $
		and for the linear strand we need
		$$
		M_s(\bk{b})=\binom{9}{s}-\binom{2}{s}-\binom{3}{s}-\binom{4}{s}.
		$$
		In particular, $
		M_2=36-1-3-6=26,
		M_3=84-0-1-4=79,$ and $
		M_9=1.
		$
		 For the higher strands we also use the block polynomials
		$$
		P_1(z)=\sum_{q=2}^{2}(q-1)\binom{2}{q}z^q=z^2,
		$$
		$$
		P_2(z)=\sum_{q=2}^{3}(q-1)\binom{3}{q}z^q=3z^2+2z^3,
		$$
		$$
		P_3(z)=\sum_{q=2}^{4}(q-1)\binom{4}{q}z^q=6z^2+8z^3+3z^4.
		$$
		Now, the linear-strand formula gives
		\begin{align*}
		 \beta_{1,2}(G)
		=& (2-1)\binom{2}{2}
		+
		\sum_{r=1}^3 (2-1)\binom{b_r}{2}
		+
		\sum_{u=1}^{1}(2-1)\binom{2}{u}\sum_{r=1}^3 \binom{b_r}{2-u}.
		\beta_{1,2}(G)\\
		=& \binom{2}{2}
		+
		\left(\binom{2}{2}+\binom{3}{2}+\binom{4}{2}\right)
		+
		\binom{2}{1}\left(\binom{2}{1}+\binom{3}{1}+\binom{4}{1}\right)=29.
		\end{align*}
		 Now for $j=3$, we have
		$$
		\beta_{2,3}(G)
		= 2\binom{2}{3}
		+
		\sum_{r=1}^3 2\binom{b_r}{3}
		+
		\sum_{u=1}^{2} 2\binom{2}{u}\sum_{r=1}^3 \binom{b_r}{3-u}
		+
		\sum_{u=1}^{1} u\binom{2}{u}M_{3-u}=120.
		$$
		 For  $\beta_{2,4}(G)$,  Theorem~\ref{thm:clique-higher}, gives
		$$
		\beta_{2,4}(G)=\sum_{u=0}^{\min{2,4-4}} \binom{2}{u}[z^{4-u}]\EE_2(z;\bk{b}).
		$$
		As $\min\{2,0\}=0$, only $u=0$ occurs, and 
		$ \beta_{2,4}(G)=[z^4]\EE_2(z;\bk{b}). $
		Now
		$$
		\EE_2(z;\bk{b})=P_1(z)P_2(z)+P_1(z)P_3(z)+P_2(z)P_3(z),
		$$
		so the coefficient of $z^4$ comes only from choosing the $z^2$-term in each factor, and we have
		$$
		[z^4]P_1P_2 = 1\cdot 3=3, ~
		[z^4]P_1P_3 = 1\cdot 6=6, ~
		[z^4]P_2P_3 = 3\cdot 6=18.
		$$
		Hence, we have
		$ \beta_{2,4}(G)=3+6+18=27. $
		
		\medskip
		\noindent
		For $\beta_{3,6}(G)$,  Theorem~\ref{thm:clique-higher}, gives
		$$
		\beta_{3,6}(G)=\sum_{u=0}^{\min{2,6-6}} \binom{2}{u}[z^{6-u}],\EE_3(z;\bk{b}).
		$$
		Only $u=0$ contributes, and since there are exactly three blocks, so
		$ \EE_3(z;\bk{b})=P_1(z)P_2(z)P_3(z). $
		Thus, we obtain
		$ \beta_{3,6}(G)=[z^6]\big(P_1P_2P_3\big). $ The only way to obtain degree $6$ is to take the $z^2$-term from each factor. Therefore, we have 
		$ \beta_{3,6}(G)=1\cdot 3\cdot 6=18. $
		
		Also, $11-8=3$, so $\beta_{8,11}(G)$ also lies on the third strand, and 
		$$
		\beta_{8,11}(G)=\sum_{u=0}^{2}\binom{2}{u}[z^{11-u}]\EE_3(z;\bk{b}).
		$$
		Since $\EE_3(z;\bk{b})=P_1P_2P_3$ has maximum degree
		$ 2+3+4=9, $
		the terms with $u=0$ and $u=1$ vanish, as they ask for the coefficients of $z^{11}$ and $z^{10}$. Hence only $u=2$ survives, and 
		$$
		\beta_{8,11}(G)=\binom{2}{2}[z^9](P_1P_2P_3)=[z^9](P_1P_2P_3).
		$$
		The top-degree term is obtained by taking the largest-degree monomial from each factor
		$ z^2\cdot 2z^3\cdot 3z^4=6z^9. $
		So, we get  $\beta_{8,11}(G)=6.$ 
		
		\medskip
		\noindent
		For $\beta_{10,11}(G)$,  $j=11$ in Theorem~\ref{thm:clique-linear}, and with $a=2$,the block part has only $9$ vertices. So, all terms involving $\binom{2}{11}$, $\binom{b_r}{11}$, $\binom{b_r}{10}$, and $\binom{b_r}{9}$ vanish. Thus, we get 
		$$
		\beta_{10,11}(G)
		= \sum_{u=1}^{2} u\binom{2}{u}M_{11-u}.
		$$
		AS
		$ M_{10}=0 ,$since  choose $10$ vertices from a set of size $9$ is absurd, while
		$$
		M_9=\binom{9}{9}-\binom{2}{9}-\binom{3}{9}-\binom{4}{9}=1,
		$$
		and hence
		$$
		\beta_{10,11}(G)
		= 1\cdot \binom{2}{1}M_{10}
		+
		2\cdot \binom{2}{2}M_9
		= 0+2\cdot 1\cdot 1=2.
		$$
		 It is worth emphasizing what these numbers are showing. The entries $\beta_{1,2}$ and $\beta_{2,3}$ come from the linear strand and are sensitive to the fact that the distinguished part is a clique. By contrast, the higher-strand values $\beta_{2,4}$, $\beta_{3,6}$, and $\beta_{8,11}$ depend only on the three block polynomials $P_1,P_2,P_3$. For that reason they coincide with the corresponding values for $\MM(2;(2,3,4))$, while the linear strand does not.
	\end{example}

	Table \ref{tab:split-clique-compare} gives the Betti numbers comparison for $\MM(2;(2,3,4))$ and $\CC(2;(2,3,4))$.
	\begin{table}[H]
		\centering
		\caption{Comparison of selected Betti numbers for $\MM(2;(2,3,4))$ and $\CC(2;(2,3,4))$.}
		\label{tab:split-clique-compare}
		\begin{tabular}{@{}cccc@{}}
			\toprule
			Bidegree $(i,j)$ & $\MM(2;(2,3,4))$ & $\CC(2;(2,3,4))$ & Observation \\ \midrule
			$(1,2)$ & $28$ & $29$ & linear strands differ \\
			$(2,3)$ & $111$ & $120$ & linear strands differ \\
			$(2,4)$ & $27$ & $27$ & higher strands coincide \\
			$(3,6)$ & $18$ & $18$ & top regularity strand coincides \\
			$(8,11)$ & $6$ & $6$ & same higher-strand tail \\
			$(10,11)$ & $1$ & $2$ & last-column corner differs \\ \bottomrule
		\end{tabular}
	\end{table}
	
	 Theorem~\ref{thm:clique-linear} extends the equal-block initial formulas from \cite{Rather2024} to arbitrary block sizes, while Theorem~\ref{thm:clique-higher} shows that the complete higher-strand pattern is already determined by the block-join component. This uniformity is not visible in the previous ad hoc calculations.
	
	\section{Regularity, extremal Betti numbers, and projective dimension}\label{sec:invariants}
	
	In this section, we extract the main homological invariants from the explicit Betti formulas.
	
	 \medskip The next theorem gives the regularity of the quotient ring for both graph families in a single closed formula.
	
	\begin{theorem}\label{thm:regularity}
		Let $G$ be either $\MM(a;\bk{b})$ or $\CC(a;\bk{b})$. Then
		$ \reg(G)=\max\{1,\nu\}, $
		where $\nu=\#\{r:b_r\geq 2\}$, and consequently, 		$ \reg(I(G))=1+\max\{1,\nu\}. $
	\end{theorem}
	
	\begin{proof}
		By Theorems~\ref{thm:split-higher} and \ref{thm:clique-higher}, the strand with difference $j-i=t$ can be nonzero only if one chooses $t$ distinct blocks each contributing at least two vertices. Therefore, no strand with $t>\nu$ can occur. On the other hand, if $\nu\geq 2$, select exactly the $\nu$ blocks of size at least $2$ and choose all their vertices. Then in the coefficient formula for $t=\nu$, the monomial $z^\sigma$ appears with coefficient 	$ \prod_{b_r\geq 2}(b_r-1)\neq 0. $
		Choosing all $a$ vertices from the distinguished part contributes the term
		$$
		\beta_{a+\sigma-\nu,\ a+\sigma}(G)\neq 0.
		$$
		Hence, $\reg(G)\geq \nu$. Since no larger difference can occur, so $\reg(G)=\nu$.
		
		If $\nu\leq 1$, then all higher strands vanish. The linear strand is nonzero,  since every nontrivial graph in our families contains an edge. Thus, $\reg(G)=1$. Combining the two cases, we have
		$$
		\reg(G)=\max\{1,\nu\}.
		$$
		Finally, for any squarefree quadratic monomial ideal, we have
		$$
		\reg(I(G))=\reg(R/I(G))+1.
		$$
		And the second formula follows.
	\end{proof}
	
	 \medskip The next corollary translates the regularity formula into a sharp criterion for $2$-linear resolution.
	
	\begin{corollary}\label{cor:2linear}
		Let $G$ be either $\MM(a;\bk{b})$ or $\CC(a;\bk{b})$. Then the following are equivalent.
		
		\begin{enumerate}[label=\textup{(\roman*)},noitemsep]
			\item $I(G)$ has a $2$-linear resolution.
			\item $\reg(I(G))=2$.
			\item $\reg(G)=1$.
			\item $\nu\leq 1$, that is, at most one block has size at least $2$.
		\end{enumerate}
	\end{corollary}
	
	\begin{proof}
		The equivalence of (i), (ii), and (iii) is standard for quadratic monomial ideals. The equivalence with (iv) follows immediately from Theorem~\ref{thm:regularity}.
	\end{proof}
	
	 \medskip The next theorem identifies the regularity corner explicitly and shows where the top nonzero strand begins.
	
	\begin{theorem}\label{thm:reg-corner}
		Assume $\nu\geq 2$, and let $ \sigma=\sum_{b_r\geq 2} b_r. $
		Then for both $G=\MM(a;\bk{b})$ and $G=\CC(a;\bk{b})$,
		$$
		\beta_{a+\sigma-\nu,\ a+\sigma}(G)=\prod_{b_r\geq 2}(b_r-1).
		$$
		Moreover, this Betti number is extremal and realizes the regularity.
	\end{theorem}
	
	\begin{proof}
		In Theorems~\ref{thm:split-higher} and \ref{thm:clique-higher}, take $t=\nu$, choose all $\nu$ nontrivial blocks, and choose all their vertices. This contributes degree $\sigma$ in the block polynomial, with coefficient $\prod_{b_r\geq 2}(b_r-1)$. Choosing all $a$ distinguished vertices, we have
		$$
		\beta_{a+\sigma-\nu,\ a+\sigma}(G)=\prod_{b_r\geq 2}(b_r-1).
		$$
		Since $\nu$ is the largest possible value of $j-i$, no Betti number with larger difference can be nonzero. Also $a+\sigma$ is the largest internal degree in the $\nu$-th strand, since one cannot choose more than all $a$ distinguished vertices and all vertices of the nontrivial blocks. Hence, the displayed Betti number lies in the north-east corner of the nonzero region of the Betti table, and is therefore extremal. By Theorem~\ref{thm:regularity}, its difference equals the regularity.
	\end{proof}
	
	 \medskip The next theorem determines the last-column extremal entry and the projective dimension.
	
	\begin{theorem}\label{thm:pd-last}
		Let $G$ be either $\MM(a;\bk{b})$ or $\CC(a;\bk{b})$ on $N=a+\sum_{r=1}^m b_r$ vertices. Then $\pd(G)=N-1$. More precisely, the following hold.
		
		\begin{enumerate}[label=\textup{(\roman*)}]
			\item If $G=\MM(a;\bk{b})$ and $m\geq 2$, then $ \beta_{N-1,N}(G)=1. $
			If $m=1$, then $ \beta_{N-1,N}(G)=b_1. $
			
			\item If $G=\CC(a;\bk{b})$ and $m\geq 2$, then $ \beta_{N-1,N}(G)=a. $
			If $m=1$, then 	$ \beta_{N-1,N}(G)=a+b_1-1. $
		\end{enumerate}
	\end{theorem}
	
	\begin{proof}
		By Hochster's formula, we have
		$$
		\beta_{N-1,N}(G)=\dim_{\kk}\wt H_0(\ind(G);\kk)=\comp(\ind(G))-1.
		$$
		Thus the problem reduces to counting connected components of the full independence complex.
		 For $G=\MM(a;\bk{b})$, Proposition~\ref{prop:decomp} implies that
		$$
		\ind(G)=\langle A\rangle \sqcup (D_{b_1}*\cdots *D_{b_m}).
		$$
		If $m\geq 2$, the join component is connected, so $\ind(G)$ has exactly two connected components, whence $\beta_{N-1,N}(G)=1$. If $m=1$, then the join component is just $D_{b_1}$, so the total number of components is $b_1+1$, thereby giving $\beta_{N-1,N}(G)=b_1$.
		 For $G=\CC(a;\bk{b})$, Proposition~\ref{prop:decomp} implies us
		$$
		\ind(G)=D_a \sqcup (D_{b_1}*\cdots *D_{b_m}).
		$$
		If $m\geq 2$, the join component is connected, so $\ind(G)$ has $a+1$ connected components, whence $\beta_{N-1,N}(G)=a$. If $m=1$, then the second component is $D_{b_1}$, so the total number of components is $a+b_1$, and $\beta_{N-1,N}(G)=a+b_1-1$. In every case $\beta_{N-1,N}(G)\neq 0$, so the projective dimension equals $N-1$.
	\end{proof}
	
	 \medskip The next proposition shows that the Betti numbers are monotone with respect to enlarging blocks.
	
	\begin{proposition}\label{prop:monotone}
		Fix $a\geq 1$ and let $\bk{b}=(b_1,\dots,b_m)$ and $\bk{c}=(c_1,\dots,c_m)$ satisfy $b_r\leq c_r$ for every $r$. Then for every bidegree $(i,j)$ and for both graph families,
		$$
		\beta_{i,j}(\MM(a;\bk{b}))\leq \beta_{i,j}(\MM(a;\bk{c})),
		\qquad
		\beta_{i,j}(\CC(a;\bk{b}))\leq \beta_{i,j}(\CC(a;\bk{c})).
		$$
	\end{proposition}
	
	\begin{proof}
		Every formula in Sections~\ref{sec:split} and \ref{sec:clique} is a sum of nonnegative multiples of binomial coefficients in the variables $b_r$. Replacing $b_r$ by a larger $c_r$ can only increase these coefficients. Therefore, each graded Betti number is monotone under coordinatewise enlargement of the block sizes.
	\end{proof}
	
	\medskip
	
	\begin{example}\label{ex:equal-block-recovery}
		Consider the equal-block examples $ \MM(3;(3,3,3,3,3))$ and $	\CC(3;(3,3,3,3,3)). $
		Here $N=18$, $\nu=5$, and $\sigma=15$. By Theorem~\ref{thm:regularity}, we obtain 
		$ \reg(R/I)=5 $  and $\reg(I)=6. $ By Theorem~\ref{thm:reg-corner}, we have  $ \beta_{13,18}=2^5=32, $ and by Theorem~\ref{thm:pd-last}, we have
		$$
		\beta_{17,18}(\MM(3;(3,3,3,3,3)))=1,
		\qquad
		\beta_{17,18}(\CC(3;(3,3,3,3,3)))=3.
		$$
		These values agree with the published Macaulay2 tables in \cite{Rather2024}. The same comparison also recovers
		$ \beta_{1,2}=60,\ \beta_{2,3}=415,\ \beta_{3,4}=1770 $
		for the split-like graph and
		$ \beta_{1,2}=63,\ \beta_{2,3}=462,\ \beta_{3,4}=2115 $
		for the clique-star graph.
	\end{example}
	
	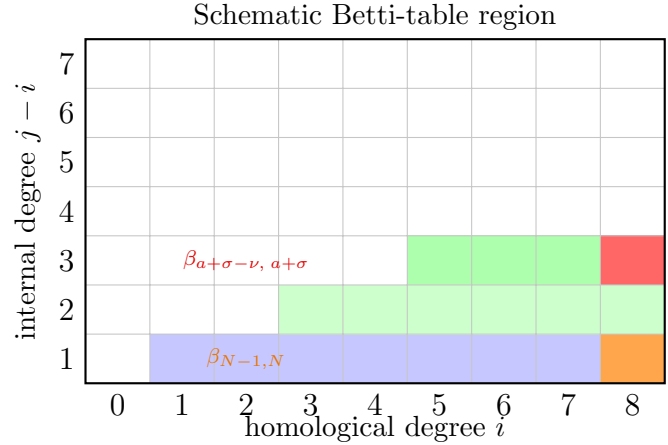
\begin{figure}[H]
		\centering
		\begin{tikzpicture}[x=0.85cm,y=0.65cm]
			
			\foreach \x in {0,...,8}{
				\foreach \y in {0,...,6}{
					\fill[white] (\x,\y) rectangle (\x+1,\y+1);
				}
			}
			
			\fill[blue!22]  (1,0) rectangle (9,1);
			\fill[green!20] (3,1) rectangle (9,2);
			\fill[green!32] (5,2) rectangle (9,3);
			
			\fill[orange!70] (8,0) rectangle (9,1);
			\fill[red!60]    (8,2) rectangle (9,3);
			
			\draw[step=1,gray!45,thin] (0,0) grid (9,7);
			\draw[thick] (0,0) rectangle (9,7);
			
			\node[font=\small] at (4.5,7.45) {Schematic Betti-table region};
			
			\node[left] at (0,0.5) {$1$};
			\node[left] at (0,1.5) {$2$};
			\node[left] at (0,2.5) {$3$};
			\node[left] at (0,3.5) {$4$};
			\node[left] at (0,4.5) {$5$};
			\node[left] at (0,5.5) {$6$};
			\node[left] at (0,6.5) {$7$};
			
			\node[below] at (0.5,0) {$0$};
			\node[below] at (1.5,0) {$1$};
			\node[below] at (2.5,0) {$2$};
			\node[below] at (3.5,0) {$3$};
			\node[below] at (4.5,0) {$4$};
			\node[below] at (5.5,0) {$5$};
			\node[below] at (6.5,0) {$6$};
			\node[below] at (7.5,0) {$7$};
			\node[below] at (8.5,0) {$8$};
			
			\node[below=8pt] at (4.5,0) {\small homological degree $i$};
			\node[rotate=90] at (-0.9,3.5) {\small internal degree $j-i$};
			
			\node[orange!90!black,font=\scriptsize,align=center] at (2.5,0.5)
			{$\beta_{N-1,N}$};
			
			\node[red!90!black,font=\scriptsize,align=center] at (2.5,2.5)
			{$\beta_{a+\sigma-\nu,\;a+\sigma}$};
			
		\end{tikzpicture}
		\caption{The two relevant extremal corners, the regularity corner from Theorem~\ref{thm:reg-corner} and the last-column corner from Theorem~\ref{thm:pd-last}.}
		\label{fig:extremal-corners}
	\end{figure}
	Figure \ref{fig:extremal-corners} shows regularity corner for Theorems~\ref{thm:reg-corner} and ~\ref{thm:pd-last}. Table \ref{tab:published-compare} gives the formulas recovered (from \cite{Rather2024}) as special  cases of our results.
	
	\begin{table}[H]
		\centering
		\caption{Published equal-block values recovered by the present formulas.}
		\label{tab:published-compare}
		\begin{tabular}{@{}llll@{}}
			\toprule
			Graph & Invariant from this paper & Value & Published source \\ \midrule
			$\MM(3;(3,3,3,3,3))$ & $\beta_{1,2},\beta_{2,3},\beta_{3,4}$ & $60,415,1770$ & \cite{Rather2024} \\
			$\MM(3;(3,3,3,3,3))$ & $\beta_{13,18},\beta_{17,18}$ & $32,1$ & \cite{Rather2024} \\
			$\CC(3;(3,3,3,3,3))$ & $\beta_{1,2},\beta_{2,3},\beta_{3,4}$ & $63,462,2115$ & \cite{Rather2024} \\
			$\CC(3;(3,3,3,3,3))$ & $\beta_{13,18},\beta_{17,18}$ & $32,3$ & \cite{Rather2024} \\ \bottomrule
		\end{tabular}
	\end{table}

	Theorem~\ref{thm:regularity} and Theorem~\ref{thm:reg-corner} are general results for the homological invariants of $\MM(a;\bk{b})$ and $\CC(a;\bk{b})$. Earlier work located only last-column extremal entries in special equal-block situations \cite{Rather2024}. Here the regularity corner is identified for arbitrary block sizes, and Corollary~\ref{cor:2linear} gives a complete and sharp criterion for $2$-linear resolution in these families.
	
	\section{Applications to algebraic graphs from groups}\label{sec:groups}
	We now illustrate the general formulas on graph classes attached to finite groups. We begin with pineapple graphs, because prime-power dihedral power graphs fall into that class.
	
	 Let $P_m^r$ denote the pineapple graph obtained from a clique $K_m$ by attaching $r$ pendant vertices to one distinguished vertex of the clique.
	
	\medskip The next theorem computes the entire Betti table of the pineapple graph, not just its linear strand.
	
	\begin{theorem}\label{thm:pineapple}
		Let $G=P_m^r$ with $m\geq 2$ and $r\geq 1$. Then all nonzero graded Betti numbers lie on the linear strand, and for every $j\geq 2$,
		$$
		\beta_{j-1,j}(G)
		= (j-1)\binom{m}{j}
		+
		\binom{r}{j-1}
		+
		\sum_{\substack{u+v=j-1\\ u,v\geq 1}}
		\binom{m-1}{u}\binom{r}{v}.
		$$
		Moreover,
		$ \beta_{i,j}(G)=0$ for $j-i\geq 2, $ $ \reg(G)=1,  \reg(I(G))=2,  \pd(G)=m+r-1, $
		and
		$
		\beta_{m+r-1,m+r}(G)=1
		$
		is the unique extremal Betti number.
	\end{theorem}
	
	\begin{proof}
		Let $v_0$ be the distinguished clique vertex, let $C'=\{u_1,\dots,u_{m-1}\}$ be the remaining clique vertices, and let $L=\{\ell_1,\dots,\ell_r\}$ be the leaves. Every induced subcomplex of $\ind(G)$ is a disjoint union of contractible pieces and isolated vertices, so higher reduced homology vanishes. Thus all nonzero Betti numbers lie on the linear strand.
		
		To compute $\beta_{j-1,j}(G)$, use Hochster's formula and count disconnected induced subcomplexes on $j$ vertices.
		
		\smallskip
		
		\noindent\textbf{Case 1:} If $W$ lies entirely in the clique $K_m={v_0}\cup C'$. Then $\ind(G)_W\cong D_j$, contributing $(j-1)\binom{m}{j}$.
		
		\smallskip
		
		\noindent\textbf{Case 2:} If $W=\{v_0\}\cup T$ with $T\subseteq L$ and $|T|=j-1$. Then
		$$
		\ind(G)_W\cong {v_0}\sqcup \langle T\rangle,
		$$
		so $\dim_{\kk}\wt H_0(\ind(G)_W;\kk)=1$, and its contribution is $\binom{r}{j-1}$.
		
		\smallskip
		
		\noindent\textbf{Case 3:} If $W=\{v_0\}\cup U\cup T$, where $U\subseteq C'$, $T\subseteq L$, $u=|U|\geq 1$, $v=|T|\geq 1$, and $u+v=j-1$. Then $\ind(G)_W$ is the disjoint union of the isolated vertex ${v_0}$ and a contractible component on $U\cup T$, and  hence $\dim_{\kk}\wt H_0(\ind(G)_W;\kk)=1$. The number of such subsets is $\binom{m-1}{u}\binom{r}{v}$, and summing over $u+v=j-1$ gives the third term.
		
		\smallskip
		
		All other subsets give connected induced subcomplexes, and hence contribute zero. This proves the linear formula. Since there are no higher strands, $\reg(G)=1$ and $\reg(I(G))=2$. The full independence complex has exactly two connected components, namely the isolated facet $\{v_0\}$ and the contractible component on $C'\cup L$. Therefore, we obtain 
		$ \beta_{m+r-1,m+r}(G)=\comp(\ind(G))-1=1, $
		which is extremal, and $\pd(G)=m+r-1$.
	\end{proof}
	
	\medskip
	
	Recall that the undirected power graph $P(\Gamma)$ of a finite group $\Gamma$ has vertex set $\Gamma$ and two distinct vertices adjacent whenever one is a positive power of the other \cite{ChakrabartyGhoshSen2009,AbawajyKelarevChowdhury2013}.

	An important class comes from elementary abelian groups. If $\Gamma$ has order $p^z$ and every nonidentity element has order $p$, then the power graph is a join of one distinguished vertex with many copies of $K_{p-1}$, see \cite{ChelvamSattanathan2013,Rather2024}.
	
	\medskip The next corollary gives a complete homogeneous formula for elementary abelian groups.
	
	\begin{corollary}\label{cor:elementary}
		Let $\Gamma$ be an elementary abelian group of order $p^z$, and let
		$ \ell=\frac{p^z-1}{p-1}. $
		Then
		$ P(\Gamma)\cong \CC(1;(\underbrace{p-1,\dots,p-1}_{\ell\text{ times}})),$ and for   $j\geq 2$,
		$$
		\beta_{j-1,j}(P(\Gamma))
		= \ell (j-1)\binom{p-1}{j}
		+
		\ell (j-1)\binom{p-1}{j-1}
		+
		\binom{\ell(p-1)}{j-1}
		- \ell \binom{p-1}{j-1},
		$$
		and for $t\geq 2$,
		$$
		\beta_{j-t,j}(P(\Gamma))
		= \binom{\ell}{t}[z^j]P_{p-1}(z)^t
		+
		\binom{\ell}{t}[z^{j-1}]P_{p-1}(z)^t,
		$$
		where
		$$
		P_{p-1}(z)=\sum_{q=2}^{p-1}(q-1)\binom{p-1}{q}z^q.
		$$
		In particular,
		$ \reg(R/I(P(\Gamma)))=\max\{1,\ell\},$ and $
		\pd(R/I(P(\Gamma)))=p^z-1. $
	\end{corollary}
	
	\begin{proof}
		By the standard subgroup count, $\Gamma$ has exactly $\ell=\tfrac{p^z-1}{p-1}$ subgroups of order $p$, each contributing a clique $K_{p-1}$ among the nonidentity vertices, and the identity is adjacent to every vertex. Hence $P(\Gamma)\cong K_1*\ell K_{p-1}$, which is exactly $\CC(1;(\underbrace{p-1,\dots,p-1}_{\ell}) )$; see \cite{ChelvamSattanathan2013,Rather2024}. Applying Theorems~\ref{thm:clique-linear}, \ref{thm:clique-higher}, and \ref{thm:pd-last} with $a=1$ and all $b_r=p-1$ yields the formulas.
	\end{proof}
	
	Finally, we return to the dihedral prime-power case already visible in \cite{Rather2024}.
	
	\medskip The next corollary shows that prime-power dihedral power graphs are pineapple graphs, so their full Betti tables are known.
	
	\begin{corollary}\label{cor:dihedral}
		Let $\Gamma=D_{2n}$ be the dihedral group of order $2n$, and assume $n=p^\alpha$ is a prime power. Then
		$ P(\Gamma)\cong P_n^n, $ and consequently
		$$
		\beta_{j-1,j}(P(\Gamma))
		= (j-1)\binom{n}{j}
		+
		\binom{n}{j-1}
		+
		\sum_{\substack{u+v=j-1\\ u,v\geq 1}}
		\binom{n-1}{u}\binom{n}{v},
		$$
		all higher strands vanish, with 
		$ \reg(R/I(P(\Gamma)))=1,$ and $ \pd(R/I(P(\Gamma)))=2n-1. $
	\end{corollary}
	
	\begin{proof}
		The structural description $P(D_{2n})\cong P_n^n$ in the prime-power case is exactly the pineapple description obtained in \cite{Rather2024}. The cyclic subgroup of rotations forms a clique of size $n$, and the $n$ reflections become leaves attached to the identity. The formulas therefore follow from Theorem~\ref{thm:pineapple}.
	\end{proof}

	\begin{example}\label{ex:group-example}
		For the elementary abelian group of order $9$, Proposition~\ref{cor:elementary} gives a complete Betti-table profile with both linear and higher strands:
		$$
		\beta_{1,2}=12,\ \beta_{2,3}=32,\ \beta_{3,4}=56,\ \beta_{2,4}=6,\ \beta_{3,5}=6,\ \beta_{4,7}=4,\ \beta_{8,9}=1.
		$$
		Thus even this small group graph already displays a genuinely multi-strand resolution.
	\end{example}
	
	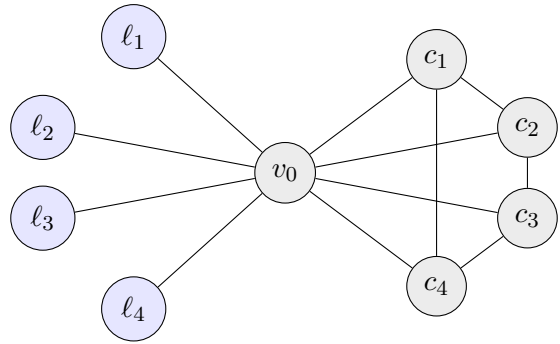
\begin{figure}[H]
		\centering
		\begin{tikzpicture}[scale=1, every node/.style={font=\small,circle,draw,minimum size=6mm}]
			\node[fill=gray!15] (v0) at (0,0) {$v_0$};
			\node[fill=gray!15] (c1) at (2,1.5) {$c_1$};
			\node[fill=gray!15] (c2) at (3.2,0.6) {$c_2$};
			\node[fill=gray!15] (c3) at (3.2,-0.6) {$c_3$};
			\node[fill=gray!15] (c4) at (2,-1.5) {$c_4$};
			
			\node[fill=blue!10] (l1) at (-2,1.8) {$\ell_1$};
			\node[fill=blue!10] (l2) at (-3.2,0.6) {$\ell_2$};
			\node[fill=blue!10] (l3) at (-3.2,-0.6) {$\ell_3$};
			\node[fill=blue!10] (l4) at (-2,-1.8) {$\ell_4$};
			
			\draw (v0)--(c1)--(c2)--(c3)--(c4)--(c1);
			\draw (v0)--(c2);
			\draw (v0)--(c3);
			\draw (v0)--(c4);
			\draw (v0)--(l1);
			\draw (v0)--(l2);
			\draw (v0)--(l3);
			\draw (v0)--(l4);
		\end{tikzpicture}
		\caption{A pineapple graph $P_4^4$, isomorphic to the power graph of $D_8$.}
		\label{fig:pineapple}
	\end{figure}
	Figure \ref{fig:pineapple} shows a pineapple graph $P_{4}^{4}$, and Theorem~\ref{thm:pineapple} gives its full Betti table of $P(D_8)$. Table \ref{tab:group-apps} gives some homological invariants of some power graphs of groups, where $E_9\cong \mathbb{Z}_{3}\times \mathbb{Z}_{3}$ denotes the elementary abelian group of order $9$.
	\begin{table}[H]
		\centering
		\caption{Sample applications to group graphs.}
		\label{tab:group-apps}
		\begin{tabular}{@{}llll@{}}
			\toprule
			Graph & Structural model & Selected Betti numbers & Invariants \\ \midrule
			$P(\ZZ_8)$ & $K_8$ & $\beta_{1,2}=28$, $\beta_{7,8}=1$ & $\reg R/I=1$, $\pd=7$ \\
			$P(E_9)$ & $K_1*4K_2$ & $\beta_{1,2}=12$, $\beta_{2,4}=6$, $\beta_{8,9}=1$ & $\reg R/I=4$, $\pd=8$ \\
			$P(D_8)$ & $P_4^4$ & $\beta_{1,2}=10$, $\beta_{3,4}=37$, $\beta_{7,8}=1$ & $\reg R/I=1$, $\pd=7$ \\ \bottomrule
		\end{tabular}
	\end{table}
	
	\medskip The next problem isolates the main case that remains open after the present paper.
	
	\begin{openproblem}\label{op:cyclic-general}
		Determine the complete graded Betti table of $R/I(P(\ZZ_n))$ for arbitrary $n$ when the power graph is not covered by the complete, pineapple, or block-join models considered here.
	\end{openproblem}
	
	 The group applications above both recover and extend earlier results. Corollary \ref{cor:dihedral} recover special formulas from \cite{Rather2024,RatherWang2026}, but Corollary~\ref{cor:elementary} is stronger than the equal-block initial-strand statements in \cite{Rather2024}, since it gives the entire Betti table, regularity, and projective dimension in one stroke.

	\section{Conclusion}\label{sec:conclusion}
	
	This paper closes the main gap left by the recent literature on Betti numbers of split-like algebraic graphs and power graphs of finite groups. The central contribution is a complete computation of the graded Betti numbers of the heterogeneous graph families
	$ \MM(a;\bk{b})=\overline{K}_a * \bigsqcup_{r=1}^m K_{b_r} $ and $\CC(a;\bk{b})=K_a * \bigsqcup_{r=1}^m K_{b_r}, $ with arbitrary block sizes. The method rests on a structural decomposition of the independence complex into a simplex or discrete component together with an iterated join of discrete complexes. Once this decomposition is identified, Hochster's formula yields exact coefficient-extraction formulas for every graded Betti number. From these formulas we derived closed expressions for the linear strand, every higher strand, the Hilbert series, the regularity, the projective dimension, and two natural classes of extremal Betti numbers.
	
	The results recover the initial equal-block formulas from \cite{Rather2024} but go substantially further by determining the complete Betti table and the regularity. The paper also clarifies the role of the number $\nu$ of nontrivial blocks, as it is exactly the quotient regularity unless $\nu=0$, in which case the regularity is $1$. This leads to a sharp $2$-linear criterion for the edge ideal. In addition, the top regularity corner and the last-column corner are identified explicitly, making the global shape of the Betti table transparent.
	
	The group-theoretic applications show that the formulas are not merely abstract generalizations. Pineapple graphs, cyclic prime-power power graphs, elementary abelian power graphs, and prime-power dihedral power graphs all fit naturally into the framework. In particular, the present work gives complete Betti tables for classes for which only partial formulas or initial strands were previously available.
	
	There are also limitations. The power graphs of arbitrary cyclic groups, and more generally of arbitrary finite groups, need not decompose into the block-join patterns studied here. For those graphs the subgroup lattice may force more complicated simplicial topology and genuinely new homological phenomena. Thus the main open direction is to identify additional structural templates beyond the split--join families treated in this paper. A second direction is to investigate multigraded refinements, Alexander dual interpretations, and asymptotic behavior under iterated joins or subgroup-growth operations. A third direction is to extend the present coefficient-extraction method to other algebraic graph classes, especially commuting and comaximal graphs whose block structures suggest partial compatibility with the methods developed here.
	
	\section*{Declarations}
	\noindent \textbf{Data Availability:} There is no data associated with this article.
	
	\noindent \textbf{Funding:} The authors did not receive support from any organization for the submitted work.
	
	\noindent \textbf{Conflict of interest:} The authors have no competing interests to declare that are relevant to the content of this article.
	
	\noindent \textbf{Acknowledgement:} We utilized online TeX tools to generate the TikZ code for the figures.
	
	\noindent\textbf{Note:} For any comments and suggestions regarding this article, please feel free to contact at \href{mailto:bilalahmadrr@gmail.com}{bilalahmadrr@gmail.com}.
	

\end{document}